\documentclass[final,1p,times]{elsarticle}

\usepackage[colorlinks=true]{hyperref}
\usepackage{amsfonts}
\usepackage{mathrsfs}
\usepackage{amsthm,amsmath,amssymb,mathrsfs,amsfonts,graphicx,graphics,latexsym,exscale,cmmib57,dsfont,bbm,amscd,ulem}
\usepackage{color,soul}
\newtheorem{thm}{Theorem}[section]
\newtheorem{cor}[thm]{Corollary}
\newtheorem{lem}[thm]{Lemma}

\theoremstyle{definition}
\newtheorem{defn}[thm]{Definition}
\newtheorem{que}[thm]{Question}
\newtheorem{remk}[thm]{Remark}

\newcommand{\bH}{\boldsymbol{H}}
\journal{xxx}

\begin{document}

\begin{frontmatter}
\title{On universal minimal proximal flows of topological groups}
\author{Xiongping Dai}
\ead{xpdai@nju.edu.cn}
\address{Department of Mathematics, Nanjing University, Nanjing 210093, People's Republic of China}

\author{Eli Glasner}
\ead{glasner@math.tau.il}
\address{Department of Mathematics, Tel Aviv University, Tel Aviv, Israel}

\begin{abstract}
In this paper,
we show that the
action of
a characteristically simple,
non-extremely amenable (non-strongly amenable, non-amenable) group
on its universal minimal (minimal proximal, minimal strongly proximal) flow
is effective.
We present necessary and sufficient conditions, for the
action of a topological group with trivial center
on its universal minimal proximal flow, to be effective.
A theorem of Furstenberg
about the isomorphism of the universal minimal proximal flows of a discrete group
and its subgroups of finite index (\cite[Theorem~II.4.4]{G76})
is strengthened.
Finally, for a pair of groups $H < G$ the same method is applied in order to extend the action
of $H$ on its universal minimal proximal flow to an action of its commensurator group
$\mathrm{Comm}_G(H)$.
\end{abstract}

\begin{keyword}
Universal minimal proximal flow  $\cdot$ Effective action $\cdot$ Strong/extreme amenability

\medskip
\MSC[2010] 37B05
\end{keyword}
\end{frontmatter}

\section{Introduction}\label{sec1}
By a topological group we mean a group $G$ endowed with a T$_2$-topology such that the algebraic operation $(x,y)\mapsto xy^{-1}$ be continuous from $G\times G$ to $G$. Unless stated otherwise, a \textit{flow} is a triple $(G,X,\pi)$, sometimes write $(G,X)$, where $X$ is a compact T$_2$-space, $G$ is a topological group with the neutral element $e$, and where the action map $\pi\colon G\times X\rightarrow X$ is such that
\begin{itemize}
\item $\pi\colon(t,x)\mapsto tx$ is jointly continuous; $\pi_{t}\circ\pi_s=\pi_{ts}\ \forall s,t\in T$; and
 $\pi_e=\textit{id}_X$ the identity map of $X$ to itself.
\end{itemize}
A flow $(G,X,\pi)$ is called \textit{effective} if $tx=x$ for all $x\in X$ implies $t=e$; and we say $(G,X,\pi)$ is 
\textit{free}
if $t\in T$ with $t\not=e$ implies $tx\not=x$ for every $x\in X$ (see e.g. \cite{Aus,GU}).
Given any $t\in T$ with $t\not=e$, we will say that $(G,X,\pi)$ is \textit{effective at $t$} if
$\pi_t\not=\textit{id}_X$,
and \textit{free at $t$} if $tx\not=x$ for every $x\in X$.

By $\mathrm{Aut}(G)$ we will denote the group of
topological automorphisms of $G$ and we will
write $\mathrm{Aut}(G)t=\{a(t)\,|\,a\in\mathrm{Aut}(G)\}$.
Let $\mathfrak{C}(G)$ be the center of the group $G$; that is,
\begin{gather*}\mathfrak{C}(G)=\{t\in G\,|\,tg=gt\ \ \forall g\in G\}.
\end{gather*}
We will show that the
groups
$\mathrm{Aut}(G)$ and $\mathfrak{C}(G)$ are important for the dynamics of the universal flows with phase group $G$.



\begin{defn}\label{def1.1}
A topological group $G$ is said to be \textit{(topologically) characteristically simple} if $\langle\mathrm{Aut}(G)t\rangle$ is dense in $G$ for all $t\in G$ with $t\not=e$, where  $\langle\mathrm{Aut}(G)t\rangle$ is the subgroup generated algebraically by $\mathrm{Aut}(G)t$.
\end{defn}

Of course every (topologically) simple group is (topologically) characteristically simple. But,
for example, for the topological group of rational numbers $\mathbb{Q}$, equipped with the topology it
inherits from $\mathbb{R}$, we have $\mathrm{Aut}(\mathbb{Q}) = \mathbb{Q}^* = \{t \in \mathbb{Q}\,|\, t \not=0\}$,
so clearly this abelian group is characteristically simple.
Similarly, the groups $\mathbb{R}^d, \ d \ge 2$ are obviously characteristically simple.
In \cite{CN} the reader can find a classification of the locally compact, compactly generated,
characteristically simple groups.


A flow $(G,X,\pi)$ is said to be \textit{minimal} if and only if there is no proper invariant closed subset of $X$---that is, $\textrm{cls}_XGx=X$ for all $x\in X$. Since G.~Birkhoff 1927~\cite{B}, minimal flows
play a central role in the theory of topological dynamical systems (cf.~\cite{F67, E69, G76, V77, Aus} etc.).
In this paper, we will be mainly concerned with the effectiveness of the universal minimal
and the universal minimal (strongly) proximal
flows of a general topological group $G$.

In Section \ref{sec3} we show that the
action of a characteristically simple,
non-extremely amenable (non-strongly amenable, non-amenable) group
on its universal minimal (minimal proximal, minimal strongly proximal) flow
is effective.

In Section \ref{sec4}
we present necessary and sufficient conditions, for the
action of a topological group with trivial center
on its universal minimal proximal flow, to be effective.

In Section \ref{sec5} we strengthen
a theorem of Furstenberg
about the isomorphism of the universal minimal (strongly) proximal flows of a discrete group $G$ and its subgroups of finite index.

Finally, in Section \ref{Comm}, for a pair of groups $H < G$ the same method is applied
in order to extend the action of $H$ on its universal minimal proximal flow to an action of its commensurator group
$\mathrm{Comm}_G(H)$.


\section{Statements of the main results}
\subsection{Universal minimal flows}\label{sec2.1}
Let $G$ be any topological group. Recall that a minimal flow $(G,X,\pi)$ is called a \textit{universal minimal flow of $G$} if any minimal flow $(G,Y,\psi)$ is a factor of $(G,X,\pi)$; that is, there exists a (not necessarily unique) continuous map $h$
from
$X$ onto $Y$ with $h(\pi(t,x))=\psi(t,h(x))$ for all $t\in G$ and $x\in X$, written as
$(G,X,\pi)
\xrightarrow{h}(G,Y,\psi)$.
See, e.g., \cite{E60, V77, Aus}.
\begin{itemize}
\item[(A)] \textit{Given any topological group $G$, there exists a unique (up to isomorphism) universal minimal flow $(G,\mathrm{M}(G))$.} (See, e.g.,~\cite[Corollary~1 
and Theorem~2]{E60}, \cite[Theorem~8.1]{Aus}, and \cite{GL}.) This theorem will be generalized to topological semigroups; see Theorem~\ref{thm3.5}.

\item[(B)]
\textit{If $G$ is a locally compact group, then its universal minimal flow is
free.}
(See \cite[Theorem~3]{E60} and \cite[Theorem~8.3]{Aus} for $G$ a discrete group;
and \cite[Theorem~2.2.1]{V77} for any locally compact group.
See also \cite[Section 3.3]{P}.)
\end{itemize}
However,
for non-locally compact groups
the question whether the universal minimal flow is
free, is an interesting one and has drawn a lot of attention in the last decade;
see e.g. \cite{P}.
For example, when $G$ is \textit{extremely amenable} (i.e., the universal minimal flow of $G$
is trivial with one-point phase space; cf., e.g.,~\cite{G98}), then the universal minimal flow of $G$ is
of course not effective.

In the literature~\cite{E60, Aus, V77}, the freeness is usually proven by using the
$\beta$-compactification
of $G$.
Using a different approach,
established originally for proving the effectiveness of the universal minimal proximal flow
in \cite{G76}, we will show in $\S\ref{sec3.1}$ the following
result.

\begin{thm}\label{thm2.1}
Let $G$ be a topological group which is not extremely amenable; and let $(G,\mathrm{M}(G))$ be its universal minimal flow. Then
$(G,\mathrm{M}(G))$ is effective at every $t\in G$ with $\mathrm{cls}_G^{}\langle\mathrm{Aut}(G)t\rangle=G$.
\end{thm}

In particular, if the canonical flow $\mathrm{Aut}(G)\times G\rightarrow G$ by $(a,t)\mapsto a(t)$ is transitive at $t\in G$, then either $G$ is extremely amenable or $(G,\mathrm{M}(G))$ is effective at $t$.

It is interesting to note that our effectiveness condition of the universal minimal flow of a topological group is in fact independent of the phase space.

\begin{cor}\label{cor2.2}
Let $G$ be a
characteristically simple
group which is not extremely amenable; then $G$ acts effectively on its universal minimal space $\mathrm{M}(G)$.
\end{cor}

A {\it semiflow} is a triple
$(T,X,\pi)$ where $X$ is a compact T$_2$-space, $T$ a topological semigroup
with a neutral element $e$ (a monoid), and the action map $\pi\colon G\times X\rightarrow X$ satisfies the properties
\begin{itemize}
\item $\pi\colon(t,x)\mapsto tx$ is jointly continuous; $\pi_{t}\circ\pi_s=\pi_{ts}\ \forall s,t\in T$; and

\item $\pi_e=\textit{id}_X$ the identity map of $X$ to itself.
\end{itemize}
Let $(T,X,\pi)$ and $(T,Y,\psi)$ be two semiflows; a continuous surjection $X\xrightarrow{h}Y$ is called a homomorphism from $(T,X,\pi)$ onto $(T,Y,\psi)$, written as $(T,X,\pi)\xrightarrow{h}(T,Y,\psi)$, if
$h(\pi(t,x))=\psi(t,h(x))$ for all $t\in T$ and $x\in X$.

The idea of Theorem~\ref{thm2.1}
is also useful for universal minimal semiflows of
characteristically simple
semigroups (with the obvious definition); see Theorem~\ref{thm3.6}, Theorem~\ref{thm3.7} and Corollary~\ref{cor3.8} in $\S\ref{sec3.2}$.

Of course proving that a flow is free is a much stronger result than showing that
it is effective; so in that respect our effectiveness results yield nothing new for locally compact
groups. Nonetheless, we hope that for non-locally compact acting groups this approach
will become useful.

\subsection{Universal minimal proximal flows}\label{sec2.2}
A flow $(G,X,\pi)$ is called \textit{proximal} if for any $x,y\in X$ there is a net $\{t_n\}$ in $G$ such that $\lim t_nx=\lim t_ny$. Recall that a minimal proximal flow of a topological group $G$ is \textit{universal} if it has every minimal proximal flow with the same phase group $G$ as a factor (cf.~\cite[$\S$II.4]{G76}).
\begin{itemize}
\item[(C)] \textit{For every topological group $G$, there exists a unique (up to isomorphism) universal minimal proximal flow for it.} We will denote this universal minimal proximal flow by $(G,\Pi(G))$. (See 
\cite[Proposition~II.4.2]{G76}.)
\end{itemize}
We notice that a topological group $G$ is \textit{strongly amenable} if and only if $\Pi(G)$ is a singleton (cf.~\cite[p.~22]{G76}).
Abelian, and more generally nilpotent, groups are strongly amenable (cf.~\cite[Theorem~II.3.4]{G76}).

Let $G$ be a topological group; we denote by $\mathrm{Homeo}(\Pi(G))$
the group of
self homeomorphisms of $\Pi(G)$ in the sequel.
\begin{itemize}
\item[(D)]
\textit{There is a homomorphism $a\mapsto\hat{a}$ of $\mathrm{Aut}(G)$ into $\mathrm{Homeo}(\Pi(G))$.
The homeomorphism $\hat{a}$ satisfies the equation
$\hat{a}(tx) = a(t)\hat{a}(x)$ for every $t \in G$ and $x \in \Pi(G)$.
For each $t\in G$ it sends the inner-automorphism $\sigma_{\!t}\colon g\mapsto tgt^{-1}$ to the homeomorphism
$\widehat{\sigma_{\!t}}\colon x\mapsto tx$ of $\Pi(G)$.
The flow
$(G, \Pi(G))$ is effective if and only if
the map $t\mapsto\widehat{\sigma_{\!t}}$, from $G$ to $\mathrm{Homeo}(\Pi(G))$,
is one-to-one.}

(This is due to
Furstenberg; see
\cite[Proposition~II.4.3]{G76}.)
\end{itemize}

It is easy to check that if $(G,\Pi(G))$ is effective, then the universal minimal flow of $G$ is also effective. So we are now concerned with the effectiveness of the universal minimal proximal flow $(G,\Pi(G))$.
Since a central element of $G$ must act as the identity on $\Pi(G)$ we have to assume that
$G$ has a trivial center, $\mathfrak{C}(G) =\{e\}$. Also notice that  $\mathfrak{C}(G)$
is the kernel of the homomorphism $t \mapsto \sigma_t$ from $G$ into $\mathrm{Aut}(G)$.

We can strengthen Furstenberg's result (D)
as follows:

\begin{thm}\label{thm2.3}
Let $G$ be a topological group with $\mathfrak{C}(G)=\{e\}$. Then the following statements are pairwise equivalent.
\begin{enumerate}
\item[$(1)$] $(G,\Pi(G))$ is effective.
\item[$(2)$] The map $t\mapsto\widehat{\sigma_{\!t}}$, from $G$ to $\mathrm{Homeo}(\Pi(G))$,
is one-to-one.
\item[$(3)$] $a\mapsto\hat{a}$ of $\mathrm{Aut}(G)$ to $\mathrm{Homeo}(\Pi(G))$ is one-to-one.
\end{enumerate}
Hence if one of the above $(1), (2), (3)$ holds, then the universal minimal flow of $G$ is effective.
\end{thm}


We will prove this theorem in $\S\ref{sec4}$ following the framework established in \cite[$\S$II.4]{G76}.

It follows from
(D) (\cite[Proposition~II.4.3]{G76}) that
the universal minimal proximal flow $(G,\Pi(G))$ can be extended to a flow
$(\mathrm{Aut}(G),\Pi(G))$,
with the discrete topology of $\mathrm{Aut}(G)$,
as follows:
\begin{equation*}
\xi\colon\mathrm{Aut}(G)\times\Pi(G)\rightarrow\Pi(G);\quad (a,x)\mapsto\hat{a}(x).
\end{equation*}
However, usually
in $(\mathrm{Aut}(G),\Pi(G))$,
we can not take $\mathrm{Aut}(G)$ to be equipped with
its natural compact-open topology.

A simple application of Theorem~\ref{thm2.3} is
Corollary~\ref{cor4.3}
which says that, for $G$ with trivial center,
$(\mathrm{Aut}(G),\Pi(G))$ is effective if so is $(G,\Pi(G))$.

In particular, in a way similar to
the situation in Theorem~\ref{thm2.1},
we may consider the effectiveness of the universal minimal proximal flow of a
characteristically simple
group.

\begin{thm}\label{thm2.4}
Let $G$ be a topological group which is not strongly amenable and let
$(G,\Pi(G))$ be its universal minimal proximal flow; then the action of $G$ is effective at
every $t\in G$ with $\mathrm{cls}_G^{}\langle\mathrm{Aut}(G)t\rangle=G$.
\end{thm}

\begin{cor}\label{cor2.5}
Let $G$ be a
characteristically simple
group which is not strongly amenable; then $G$ acts effectively on $\Pi(G)$.
\end{cor}

Recall that
a topological group has the \textit{Rohlin property} if it has a dense conjugacy class
in $G \setminus \{e\}$, and the \textit{strong Rohlin property} if it has a co-meager conjugacy class in $G \setminus \{e\}$ (cf.~\cite[Definition~3.3]{GW}).
To illustrate the subject, we give here several examples of
Polish (non-locally compact) groups that have the Rohlin property:
\begin{itemize}
\item[(1)] The group $\mathrm{Aut}(X,\mathscr{X},\mu)$ of measure-preserving automorphisms of a standard measure space $(X,\mathscr{X},\mu)$ equipped with the (Polish) weak topology.

\item[(2)] The group $U(H)$ of unitary operators on a separable infinite-dimensional Hilbert space $H$ equipped with the strong operator topology.

\item[(3)] The group of homeomorphisms of the Cantor set and the group of homeomorphisms
of the Hilbert cube $[-1,1]^\mathbb{N}$, equipped with the topology of uniform convergence.
\end{itemize}
Examples of groups with the strong Rohlin property are:
\begin{enumerate}
\item[(4)] The group $S_\infty(\mathbb{N})$ of all permutations of a countable set with the topology of pointwise convergence;

\item[(5)] The group $H_+[0,1]$ of order preserving homeomorphisms of the unit interval;

\item[(6)] The group $H(X)$ of homeomorphisms of a Cantor set.
\end{enumerate}
We will say that a topological group has the \textit{characteristic Rohlin property} if it has
a dense ${\mathrm{Aut}(G)}$ orbit in $G \setminus \{e\}$ (i.e. there is some $g \in G$ with ${\mathrm{Aut}(G)} g$ dense in in $G \setminus \{e\}$),
and the \textit{characteristic strong Rohlin property} if it has
a co-meager
${\mathrm{Aut}}(G)$ orbit in $G \setminus \{e\}$.

Clearly, the (strong) Rohlin property implies the characteristic (strong) Rohlin property.
However, they are conceptually different.
For instance, the abelian group $\mathbb{Q}$ has the strong characteristic Rohlin property.

By Theorem~\ref{thm2.4}, we can easily obtain the following:

\begin{cor}\label{cor2.6}
Let $G$ be a topological group which is not strongly amenable.
If $G$ has the characteristic (strong) Rohlin property, then there is a (co-meager) dense set $E$ of $G\setminus\{e\}$ such that $G$ acts effectively on $\Pi(G)$ at each $t\in E$.
\end{cor}
See Corollary~\ref{cor4.6} for
a related freeness criterion.

\subsection{An isomorphism theorem of universal minimal proximal flows}\label{sec2.3}

The following result
strengthens a theorem of
Furstenberg (see \cite[Theorem~II.4.4]{G76}).

\begin{thm}\label{thm2.7}
Let $G$ be a topological group and $S$ a closed subgroup of finite index in $G$. Then $(S,\Pi(S))$ can be extended to $(G,\Pi(S))$ so that the flows $(G,\Pi(G))$ and $(G,\Pi(S))$ are isomorphic.
Also, $(G,\Pi(G))$ is
free
if and only if so is $(G,\Pi(S))$.
\end{thm}

We will prove this theorem in $\S\ref{sec5}$ by making use of some arguments established in \cite{G76}.
\subsection{Universal minimal strongly proximal flows}
If a topological group $G$ acts on $Q$ which is a compact convex subset of a locally convex topological vector space and if each $t\in G$ acts as an affine transformation, then we say that $(G,Q)$ is an \textit{affine flow}. See \cite{G76, Aus}.

Let $\mathcal{M}(X)$ be the set of all \textit{quasi-regular} Borel probability measures on the compact T$_2$-space $X$, which is compact and convex under the weak-* topology.
Given any flow $(T,X,\pi)$, let $(T,\mathcal{M}(X),\pi_*)$ be the naturally induced affine flow on
$\mathcal{M}(X)$.
See \cite[p.~31]{G76}.

A flow $(T,X,\pi)$ is said to be \textit{strongly proximal} if the induced affine flow $(T,\mathcal{M}(X),\pi_*)$ is proximal~\cite[p.~161]{G75} and \cite[p.~31]{G76}. It is easy to see that a subflow of a strongly proximal flow is strongly proximal and every strongly proximal flow is also proximal.
As observed in \cite[p.~32]{G76} the diagonal-wise product of strongly proximal flows is strongly proximal.
Here is a precise statement and a short proof:

\begin{itemize}
\item[(E)] \textit{Strong proximality
is preserved under diagonal-wise product of any cardinality.}

\begin{proof}
Let $\{(G,X_i)\}_{i\in I}$ be a family of strongly proximal flows and set $X=\prod_{i\in I}X_i$ and $\Pr_i\colon x=(x_i)_{i\in I}\mapsto x_i$.
Let $\mathfrak{m}\in\mathcal{M}(X)$ be any Borel probability measure. Then by \cite[Lemma~III.1.1]{G76}, we may suppose that $\mathfrak{m}$ belongs to a minimal set $M$ of $(G,\mathcal{M}(X))$.
Since $\Pr_i(M)$ is minimal for $(G,\mathcal{M}(X_i))$ and $(G,\mathcal{M}(X_i))$ is proximal,
it follows that
$\Pr_i(M)\subseteq X_i$ for all $i\in I$. This implies that $\Pr_i(\mathfrak{m})$ is a point mass for any $i\in I$ and thus $\mathfrak{m}$ itself is a point mass.
This proves the statement.
\end{proof}
\end{itemize}
Hence as in (C) before one proves that
\begin{itemize}
\item[(F)] \textit{Associated to any topological group $G$, there exists a unique, up to an isomorphism, universal minimal strongly proximal flow.} We will denote this flow by $(G,\Pi_s(G))$. See \cite[p.~32]{G76}.
\end{itemize}
In the same way, Theorem~\ref{thm2.3} and Theorem~\ref{thm2.7} can be restated for $\Pi_s$ instead of $\Pi$.
Moreover, by using the characterizations of amenable group
given in ~\cite[Theorem~III.3.1]{G76}, we can restate Theorem~\ref{thm2.4}, Corollary~\ref{cor2.5} and Corollary~\ref{cor2.6}, respectively, as follows.

\begin{thm}\label{thm2.8}
Let $G$ be a topological group which is not amenable and let
$(G,\Pi_s(G))$ be its universal minimal strongly proximal flow. Then $G$ acts effectively on $\Pi_s(G)$ at
every $t\in G$ with $\mathrm{cls}_G^{}\langle\mathrm{Aut}(G)t\rangle=G$.
\end{thm}

\begin{cor}\label{cor2.9}
Let $G$ be a
characteristically simple
group which is not amenable; then $G$ acts effectively on $\Pi_s(G)$.
\end{cor}

\begin{cor}\label{cor2.10}
Let $G$ be a topological group which is not amenable.
If $G$ has the characteristic (strong) Rohlin property, then there is a (co-meager) dense set $E$ of $G\setminus\{e\}$ such that $G$ acts effectively on $\Pi_s(G)$ at each $t\in E$.
\end{cor}


In the following table we collect the
corollaries \ref{cor2.2}, \ref{cor2.5} and \ref{cor2.9}, all of which hold under the assumption that $G$
is characteristically simple.
%
%

\begin{table}[h]
\begin{center}
\begin{tabular}{ | l | l| l | l | l | l | }
\hline
$\mathrm{M}(G) \not = pt$  &
$G$ is not extremely amenable & $G$ is effective on $\mathrm{M}(G)$ \\
\hline
$\Pi(G) \not = pt$ & $G$ is not strongly amenable
& $G$ is effective on $\Pi(G)$\\
\hline
$\Pi_s(G) \not = pt$ & $G$ is not amenable
& $G$ is effective on $\Pi_s(G)$\\
\hline
\end{tabular}

\vspace{.2cm}
\caption{Effective actions of a characteristically simple group}
\end{center}
\end{table}

\subsection{Universal irreducible affine flows}\label{sec2.5}
We now consider the affine flow $(G,\mathcal{M}(\Pi_s(G)))$ induced on the compact convex set
$\mathcal{M}(\Pi_s(G))$ of quasi-regular Borel
probability measures on the universal minimal strongly proximal flow $(G,\Pi_s(G))$.
Recall that an affine flow $(G,Q)$ is \textit{irreducible} if it contains no proper non-empty closed convex invariant subset~\cite{G76}.




%

Recall that an irreducible affine flow $(G,Q)$ is called a \textit{universal irreducible affine flow} of $G$ if for every irreducible affine flow $(T,Q^\prime)$ there exists an affine homomorphism
$(G,Q)\xrightarrow{\phi}(G,Q^\prime)$.
\begin{itemize}
\item[(G)] {\it For any topological group $G$, $(G,\mathcal{M}(\Pi_s(G)))$ is the (unique) universal irreducible affine flow.} It is strongly proximal and contains $\Pi_s(G)$ (identified with the
collection of Dirac measures, or equivalently the closed set of its extremal points) as its unique minimal set
 See \cite[Proposition~III.2.4]{G76}.
\end{itemize}

%
%
%
%

Denote by $\mathrm{AHomeo}(\mathcal{M}(\Pi_s(G)))$ the group of all affine homeomorphisms of
$\mathcal{M}(\Pi_s(G))$ onto itself.
In view of this result and because an action on $\Pi_s(G)$ determines the action on $\mathcal{M}(\Pi_s(G))$
we immediately deduce the following:

\begin{thm}\label{thm2.11}
There is a homomorphism $a\mapsto\hat{a}$ of
$\mathrm{Aut}(G)$ into $\mathrm{AHomeo}(\mathcal{M}(\Pi_s(G)))$
which, for $t\in G$,
sends the inner-automorphism $\sigma_{\!t}\colon g\mapsto tgt^{-1}$
to the affine homeomorphism $\widehat{\sigma_{\!t}}\colon x\mapsto tx$ of
$\mathcal{M}(\Pi_s(G))$.
The flow $(G, \mathcal{M}(\Pi_s(G)))$ is effective if and only if the homomorphism
$$
t\mapsto\widehat{\sigma_{\!t}},
\qquad
G\to \mathrm{AHomeo}(\mathcal{M}(\Pi_s(G))),
$$
is one-to-one.
\end{thm}

\begin{thm}\label{thm2.12}
Let $G$ be a topological group and $S$ a closed subgroup of finite index in $G$. Then $(S,\mathcal{M}(\Pi_s(S)))$ can be extended to $(G,\mathcal{M}(\Pi_s(S)))$ so that the flows $(G,\mathcal{M}(\Pi_s(G)))$ and $(G,\mathcal{M}(\Pi_s(S)))$ are isomorphic.
\end{thm}

\section{The effectiveness of some universal minimal flows}\label{sec3}
This section is devoted to proving Theorem~\ref{thm2.1}.
Our argument is also useful for the effectiveness of the universal minimal semiflow associated to
characteristically simple
semigroups such as the additive semigroup $\mathbb{R}_+$ with the usual Euclidean topology.
See Theorem~\ref{thm3.6} below.

\subsection{Group actions}\label{sec3.1}


\begin{defn}\label{def-coal}
A flow $(G,X,\pi)$ is said to be \textit{coalescent} if every endomorphism of $(G,X,\pi)$ is an automorphism (cf., e.g.,~\cite[p.~115]{Aus}).
Similarly one defines the notion of coalescence for semiflows.
\end{defn}

%

The following result is due to Ellis (see \cite{E69}).

\begin{lem}\label{lem3.2}
The universal minimal flow of any topological group is coalescent.
\end{lem}
%

\begin{proof}[Proof of Theorem~\ref{thm2.1}]
Let $(G,X,\pi)$ be the universal minimal flow of the topological group $G$ where $X$ is not one-point, namely, with $G$ not extremely amenable.

Write $\pi(t,x)=tx$ for all $t\in G$ and $x\in X$. Let $\textrm{Homeo}(X)$ be the group of
homeomorphisms of $X$ onto itself.
For each $a\in\mathrm{Aut}(G)$ define a flow $\alpha\colon G\times X\rightarrow X$ by $(t,x)\mapsto t\cdot_ax$, where
\begin{equation*}t\cdot_ax=a(t)x\quad\forall t\in G\textrm{ and }x\in X.\end{equation*}
Since $a(G)=G$, it follows that $(G,X,\alpha)$ is minimal and by the university of $(G,X,\pi)$ there is a homomorphism $(G,X,\pi)\xrightarrow{\hat{a}}(G,X,\alpha)$.
Clearly, $\hat{a}(tx)=a(t)\hat{a}(x)$ for every $t\in G$ and $x\in X$.

Now for $a^{-1}$ in place of $a\in\mathrm{Aut}(G)$, define similarly a minimal flow $(G,X,\alpha^{-1})$ and 
to obtain another homomorphism $(G,X,\pi)\xrightarrow{\widehat{a^{-1}}}(G,X,\alpha^{-1})$ such that
\begin{equation*}
\widehat{a^{-1}}(tx)=a^{-1}(t)\widehat{a^{-1}}(x)
\end{equation*}
for every $t\in G$ and $x\in X$.

Now,
the composition map $\hat{a}\circ\widehat{a^{-1}}$ satisfies
\begin{gather*}
\hat{a}\circ\widehat{a^{-1}}(tx)=\hat{a}\left(a^{-1}(t)\widehat{a^{-1}}(x)\right)=a(a^{-1}(t))\hat{a}\left(\widehat{a^{-1}}(x)\right)=t\hat{a}\circ\widehat{a^{-1}}(x),
\end{gather*}
so that $\hat{a}\circ\widehat{a^{-1}}$ is an endomorphism of $(G,X,\pi)$.
By Lemma~\ref{lem3.2}, it follows that $\hat{a}\in\textrm{Homeo}(X)$ for each $a\in\mathrm{Aut}(G)$.

Next, arguing by contradiction, assume that for some $t\in G$
with $\mathrm{cls}_G^{}\langle\mathrm{Aut}(G)t\rangle=G$,
$tx=x$ for all $x\in X$.
Then by the following commutative diagram:
\begin{equation*}
\begin{CD}
x@>>> tx\\
@V{\hat{a}}VV   @VV{\hat{a}}V\\
\hat{a}(x)@>>>t\cdot_a\hat{a}(x)
\end{CD}\quad \forall x\in X\textrm{ and }a\in\mathrm{Aut}(G),
\end{equation*}
it follows that
$$\hat{a}(tx) =a(t)\hat{a}(x)=\hat{a}(x)\quad \forall x\in X\textrm{ and }a\in\mathrm{Aut}(G). $$
Since $\hat{a}\in \textrm{Homeo}(X)$,  it follows that
$a(t)y =y$ for every $y \in X$.
Then, for any $n\ge1$
$$a_1(t)\dotsm a_n(t)y=y\quad \forall y\in X\textrm{ and }a_1,\dotsc,a_n\in\mathrm{Aut}(G).$$
Thus by $\mathrm{cls}_G^{}\langle\mathrm{Aut}(G)t\rangle=G$, it follows that $Gy=y$ and by minimality $X=\{y\}$. This contradicts the hypothesis that $X$ is not a singleton.

The proof of Theorem~\ref{thm2.1} is therefore completed.
\end{proof}


\begin{thm}\label{thm-free}
Let $G$ be a topological group such that
$\mathrm{Aut}(G)t$ is dense in $G\setminus\{e\}$ for all $t\in G$ with $t\not=e$. If $(G,\mathrm{M}(G))$ is
free
at some element $\tau\in G$, then $(G,\mathrm{M}(G))$ is
free.
\end{thm}

\begin{proof}
Let $t$ be an arbitrary element of $G$ with $t \not=e$.
By contradiction, assume that $tx_0=x_0$ for some $x_0\in \mathrm{M}(G)$.
Since $\mathrm{Aut}(G)t$ is dense in $G\setminus\{e\}$, we can choose a net $\{a_n\}$ in
$\mathrm{Aut}(G)$ such that $a_n(t)\to\tau$
and $a_n(t)\hat{a}_n(x_0)=\hat{a}_n(x_0)$. Since $\mathrm{M}(G)$ is compact, then,
passing to a subnet if necessary, we may assume
$\hat{a}_n(x_0)\to y$ in $\mathrm{M}(G)$. Thus $\tau y=y$,
which contradicts the assumption that $(G,\mathrm{M}(G))$ is 
free
at $\tau$.
\end{proof}

\subsection{Universal minimal semiflows and effectiveness}\label{sec3.2}
Given any topological semigroup $T$, there exists a unique (up to isomorphism) \textit{universal minimal semiflow} of $T$, written as $(T,\textrm{M}(T))$, as in the group case, such that if $(T,X)$ is a minimal semiflow there is a homomorphism $(T,\textrm{M}(T))\xrightarrow{\phi}(T,X)$. For this we need the semigroup version of Lemma~\ref{lem3.2}.

\begin{lem}\label{lem3.4}
Let $(T,X,\pi)$ be a semiflow with $T$ a topological semigroup.
Then there is a cardinal number $\mathrm{a}$ and a minimal subset $M$ of $(T,X^\mathrm{a},\pi^\mathrm{a})$ such that $(T,M,\pi^\mathrm{a})$ is coalescent.
\end{lem}

\begin{proof}
According to Zorn's lemma, let $C$ be a maximal almost periodic set of $(T,X)$; that is, for any finite subset $\{x_1,\dotsc,x_n\}\subseteq C$, the point $(x_1,\dotsc,x_n)$ is an almost periodic point for the diagonal-wise product semiflow $(T,X^n)$; and no other almost periodic set of $(T,X)$ properly contains it (cf.~\cite[p.~67]{Aus}).

Let $z\in X^C$ such that range $z=C$ and $z\colon C\rightarrow X$ is one to one (for example, $z_c=c$, for each $c\in C$). Let $M=\textrm{cls}_{X^C}{Tz}$, and let $z^\prime\in M$. Now $z^\prime$ is an almost periodic point for $(T,X^C)$, and so $C^\prime=\textrm{range } z^\prime$ is an almost periodic set of $(T,X)$. In fact, it is easy to verify that $C^\prime$ is also maximal, and $z^\prime\colon C\rightarrow C^\prime$ is one to one.

Now, let $\varphi$ be an endomorphism of the minimal semiflow $(T,M)$. Then $(z,\varphi(z))$ is an almost periodic point of $(T,M\times M)$, so range $z\cup\textrm{ range } \varphi(z)$ is an almost periodic set of $(T,X)$. But range $z$ and range $\varphi(z)$ are both maximal almost periodic sets of $(T,X)$, so $\textrm{range }z=\textrm{ range }\varphi(z)$.

If $\gamma$ is a permutation (bijection) of $C$, let $\gamma^*$ denote the induced automorphism of $(T,X^C)$ by $\gamma^*(z)_c=z_{\gamma(c)}$ for each $c\in C$. Define $\gamma$ by $z_{\gamma(c)}=\varphi(z)_c$, for $c\in C$. Since $\varphi(z)$, (regarded as a map of $C$ to $X$) is one to one and range $\varphi(z)=\textrm{ range }z$, $\gamma$ is a permutation of $C$ and $\gamma^*(z)=\varphi(z)$.

Since $\gamma^*=\varphi$ restricted to $Tz$ and $\textrm{cls}_{X^C}{Tz}=M$, then $\gamma^*=\varphi$ on $M$ and so $\varphi$ is an automorphism.
\end{proof}

Now we can obtain the unique universal minimal semiflow of a topological semigroup following the framework of the proofs in
\cite{E69}, ~\cite[Proposition~II.4.2]{G76} and ~\cite[Theorem~8.1]{Aus}.

\begin{thm}\label{thm3.5}
For any topological semigroup $T$, there exists a universal minimal semiflow $(T,\mathrm{M}(T),\pi)$,
and any two universal minimal semiflows for $T$ are isomorphic.
\end{thm}

\begin{proof}
Let $\mathscr{M}=\{(T,X_i,\pi_i)\,|\,i\in I\}$ be the collection of
non-isomorphic minimal semiflows with the phase semigroup $T$. Define
\begin{equation*}
X={\prod}_{i\in I}X_i\quad \textrm{and}\quad \pi\colon(t,(x_i)_{i\in I})\mapsto(tx_i)_{i\in I}
\ \textrm{ from }T\times X\textrm{ to }X.
\end{equation*}
By Lemma~\ref{lem3.4}
there is a cardinal number $\mathrm{a}$ and a minimal subset $M$ of $(T,X^\mathrm{a},\pi^\mathrm{a})$ such that $(T,M,\pi^\mathrm{a})$ is coalescent.
Obviously $(T,M,\pi^\mathrm{a})$ is a universal minimal semiflow of $T$. Suppose $(T,Z,\pi_Z)$ is another universal minimal semiflow of $T$. Then there are $T$-homomorphisms \begin{equation*}
(T,M,\pi^\mathrm{a})\xrightarrow{\phi}(T,Z,\pi_Z)
\xrightarrow{\psi}(T,M,\pi^\mathrm{a}).
\end{equation*}
Since $(T,M,\pi^\mathrm{a})$ is coalescent, $\psi\circ\phi$ and also $\phi,\psi$ are all isomorphisms.
\end{proof}

Next we will be concerned with the effectiveness of the universal minimal semiflow of some
topological semigroups including $\mathbb{R}_+=[0,+\infty)$
or $\mathbb{Q}_+ = \mathbb{Q} \cap [0, \infty)$
Of course for $\mathbb{R}_+$ this way of proving effectiveness is an overkill,
as already the action of $\mathbb{R}_+$ on the $2$-torus via an irrationally oriented line
is minimal and effective (hence free). However, for general acting non locally compact semigroup
our next results may be of interest.


Let $T$ be a topological semigroup.
By $\mathcal{E}nd\,(T)$ we denote the set of continuous self homomorphisms $a$ of $T$
such that $a(T)$ is dense in $T$. By $(T,\mathrm{M}(T))$ we denote the universal minimal semiflow of $T$.
It is easy to check that:
\begin{itemize}
\item Given any $t\in\mathbb{R}_+$ with $t\not=0$,
$\quad \mathcal{E}nd\,(\mathbb{R}_+)t=\{a(t)\,|\,a\in\mathcal{E}nd\, (\mathbb{R}_+)\}=(0,+\infty)$.
\end{itemize}
By a slight modification of the proof of Theorem~\ref{thm2.1} we can obtain the following:

\begin{thm}\label{thm3.6}
Let $T$ be a topological semigroup such that $\langle\mathcal{E}nd\,(T)t\rangle$ is dense in $T$ for all $t\in T$ with $t\not=e$. Then, either $T$ is extremely amenable or
$(T,\mathrm{M}(T))$ is effective (i.e., $t\colon x\mapsto tx$ is not the identity map for any $t\in T$ with $t\not=e$).
\end{thm}

%
%

It should be noted that even if $T$ is locally compact, Theorem~\ref{thm3.6} is already beyond the framework of Veech~\cite[Theorem~2.2.1]{V77} which is proven only for locally compact groups.

We say that a semiflow $(T,X,\pi)$ is
\textit{free at $t\in T$}
if $tx\not=x\ \forall x\in X$.
We now have a semigroup version of Theorem \ref{thm-free}:

\begin{thm}\label{thm3.7}
Let $T$ be a topological semigroup such that
$\mathcal{E}nd\,(T)t$ is dense in $T\setminus\{e\}$ for all $t\in T$ with $t\not=e$. If $(T,\mathrm{M}(T))$ is
free
at some element $\tau\in T$, then $(T,\mathrm{M}(T))$ is
free.
\end{thm}



\begin{cor}\label{cor3.8}
Let $T$ be an abelian semigroup such that $\mathcal{E}nd\,(T)t$ is dense in $T\setminus\{e\}$ for all $t\in T$ with $t\not=e$. Then $(T,\mathrm{M}(T))$ is free if $T$ is not extremely amenable.
\end{cor}

\begin{proof}
If $(T,\mathrm{M}(T))$ is free at some $\tau\in T$, then by Theorem~\ref{thm3.7} it is free. Now let there be some $\tau\in T$ with $\tau\not=e$ such that $\tau x_0=x_0$ for some point $x_0\in X$. Then $\tau tx_0=tx_0$ for any $t\in T$. Since $Tx_0$ is dense in $X$ by minimality, hence $\tau x=x$ for all $x\in X$. However, this contradicts Theorem~\ref{thm3.6}. Thus $(T,\mathrm{M}(T))$ is free.
\end{proof}
\section{The effectiveness of some universal minimal proximal flows}\label{sec4}
This section will be devoted to proving Theorems~\ref{thm2.3} and \ref{thm2.4}. Let $G$ be a topological group. First, we will need a lemma.

\begin{lem}[{See~ \cite[Lemma~II.4.1]{G76}}]\label{lem4.1}
The only endomorphism a minimal proximal $G$-flow admits is the identity automorphism.
\end{lem}

In fact we can obtain a more general uniqueness result.

\begin{lem}\label{lem4.2}
If $(G,X,\pi_X)\xrightarrow{\theta}(G,Y,\pi_Y)$ is a homomorphism (not necessarily surjective) from a minimal proximal flow $(G,X,\pi_X)$ to another proximal flow $(G,Y,\pi_Y)$, then $\theta$ is the unique homomorphism admitted from $(G,X,\pi_X)$ to $(G,Y,\pi_Y)$.
\end{lem}

\begin{proof}
Let $(G,X,\pi_X)\xrightarrow{\phi}(G,Y,\pi_Y)$ be a homomorphism. Then given any $x\in X$,
set $y_1=\theta(x)$ and $y_2=\phi(x)$.
By proximality there is a net $\{t_n\}$ in $G$ and some $y_\infty\in Y$ such that
$$\lim_n t_n y_1 = \lim_n t_n y_2 = y_\infty.$$
We can assume that the limit,  $x_\infty = \lim t_n x$ exists and then
$\theta(x_\infty)=y_\infty=\phi(x_\infty)$.
Since $(G,X)$ is minimal, we conclude that $\theta(x)=\phi(x)$.
Thus $\theta=\phi$ on $X$.
\end{proof}

Recall that $\sigma_{\!t}\colon s\mapsto tst^{-1}$ is the inner-automorphism of $G$ as in (D) in $\S\ref{sec2.2}$. Let $(G,\Pi(G),\pi)$ be the universal minimal proximal flow associated with $G$ and simply write $\pi(t,x)=tx$ for $t\in G$ and $x\in\Pi(G)$ as in (C) in $\S\ref{sec2.2}$.
We also recall the construction of the homomorphism $a\mapsto\hat{a}$ of $\mathrm{Aut}(G)$ to $\mathrm{Homeo}(\Pi(G))$ introduced in \cite[p.~23]{G76}.
In fact, this is in essence the same as the construction described above for $\mathrm{M}(G)$
in the proof of Theorem~\ref{thm2.1};
the only difference is that here we use the fact that $\Pi(G)$ admits no non-trivial endomorphisms
instead of the coalescence of $\mathrm{M}(G)$.

It should be mentioned that usually we cannot expect the continuity of the homomorphism
$a\mapsto\hat{a}$, from  $\mathrm{Aut}(G)$ to $\mathrm{Homeo}(\Pi(G))$, when the former
is equipped with its natural compact-open or pointwise convergence topologies.

\begin{proof}[Proof of Theorem~\ref{thm2.3}]
(1) $\Leftrightarrow$ (2): First note that
the map
$t\mapsto\widehat{\sigma_{\!t}}$ from $G$ to $\mathrm{Homeo}(\Pi(G))$
is a homomorphism. Then the statement easily follows from $\widehat{\sigma_{\!t}}=t$.

(1) $\Rightarrow$ (3): Let $a\in\mathrm{Aut}(G)$ be such that $\hat{a}$ is the identity map on $\Pi(G)$. Then by the equality $\hat{a}(tx)=a(t)\hat{a}(x)$, it follows that $tx=a(t)x$ for all $x\in X$ and  $t\in G$. Since $(G,\Pi(G),\pi)$ is effective, then $a(t)=t$ for every $t\in G$ and so $a=\textit{id}_G$. This shows that the homomorphism $a\mapsto\hat{a}$ is one-to-one.

(3) $\Rightarrow$ (2): Since $\mathfrak{C}(G)=\{e\}$,
the map
$t\mapsto\sigma_{\!t}$ from $G$ to $\mathrm{Aut}(G)$
is one-to-one. Thus by condition (3), it follows that $t\mapsto\widehat{\sigma_{\!t}}$ is one-to-one.

The proof of Theorem~\ref{thm2.3} is completed.
\end{proof}

\begin{proof}[Proof of Theorem~\ref{thm2.4}]
This proof is analogous to the proof of Theorem~\ref{thm2.1} and so we omit the details.
\end{proof}

Let $(\mathrm{Aut}(G),\Pi(G))$ be the canonical extension of $(G,\Pi(G))$.
Then we can easily obtain the following by Theorem~\ref{thm2.3}.

\begin{cor}\label{cor4.3}
Let $G$ be a topological group with $\mathfrak{C}(G)=\{e\}$. Then $(G,\Pi(G))$ is effective if and only if so is $(\mathrm{Aut}(G),\Pi(G))$.
\end{cor}

In many cases
the group of inner-automorphisms
$\mathrm{Inn}(G) = \{\sigma_{\!t}\,|\,t\in G\}$ is a proper subgroup of
$\mathrm{Aut}(G)$
and so Corollary~\ref{cor4.3} seems to be interesting. The ``if'' part of Corollary~\ref{cor4.3} will be generalized in Theorem~\ref{thm4.4} below.

Next we consider the inheritance of
freeness of the universal minimal flows associated to topological semigroups.

\begin{thm}\label{thm4.4}
Let $T$ be a topological semigroup and $H$ a subsemigroup of $T$. Then
\begin{enumerate}
\item[$(1)$] If $(T,\mathrm{M}(T))$ is
free,
then $(H,\mathrm{M}(H))$ is also
free.
\item[$(2)$] If $(T,\Pi(T))$ is
free,
then $(H,\Pi(H))$ is also
free.
\item[$(3)$] If $(T,\Pi_s(T))$ is
free,
then $(H,\Pi_s(H))$ is also
free.
\end{enumerate}
\end{thm}

\begin{proof}
Let $(T,\mathrm{M}(T))$ be free;
then $(H,\mathrm{M}(T))$ is also free.
Now let $X$ be an $H$-minimal subset of $\mathrm{M}(T)$ and then $(H,X)$ is obviously
free. Then by the universality of $(H,\mathrm{M}(H))$, it follows that there is a homomorphism
$(H,\mathrm{M}(H))\xrightarrow{\phi}(H,X)$. This implies that $(H,\mathrm{M}(H))$ is
free. This proves (1). We can easily prove (2) and (3) similarly.
\end{proof}



\begin{thm}\label{thm4.5}
Let $G$ be a topological group with $\mathrm{Aut}(G)t$ dense in $G\setminus\{e\}$ for some $t\in G$ with $t\not=e$. If $(G,\Pi(G))$ is
free at some $\tau\in G$, then $(G,\Pi(G))$ is free at $t$.
\end{thm}

The proof is almost verbatim the same as that of Theorem~\ref{thm-free} and thus we omit its details here.

%

The following is a simple consequence of Theorem~\ref{thm4.5}, which is comparable with Corollary~\ref{cor2.6}.

\begin{cor}\label{cor4.6}
Let $G$ be a topological group with the characteristic (strong) Rohlin property, such that $(G,\Pi(G))$ is
free at some $\tau\in G$. Then there exists a (co-meager) dense subset $E$ of $G\setminus\{e\}$ such that $G$ acts
freely on $\Pi(G)$ at each $t\in E$.
\end{cor}

We note that in the same way the above Corollary~\ref{cor4.3}, Theorem~\ref{thm4.5} and Corollary~\ref{cor4.6} can be restated for $\Pi_s(G)$ in place of $\Pi(G)$.
\section{A generalization of Furstenberg's isomorphism theorem}\label{sec5}
Based on the construction of the homomorphism $a\mapsto\hat{a}$ presented in $\S\ref{sec3}$, this section will be mainly devoted to proving Theorem~\ref{thm2.7}. For that, let $G$ be a topological group and let $S$ be a \textit{closed} proper subgroup of \textit{finite index} in $G$, unless otherwise specified.

First of all, we will need a useful lemma.

\begin{lem}[{See \cite[Lemma~II.3.2]{G76}}]\label{lem5.1}
Let $T$ be a topological group and $(T,X,\varphi)$ a minimal proximal flow.
\begin{enumerate}
\item If $T$ is a compact extension of its subgroup $L$, then $(L,X,\varphi{\upharpoonright_{\!L\times X}})$ is minimal and proximal.
\item If $L$ is a closed subgroup of finite index in $T$, then $(L,X,\varphi{\upharpoonright_{\!L\times X}})$ is minimal and proximal.
\end{enumerate}
\end{lem}

\noindent
Here $T$ is called a \textit{compact extension} of $L$ if $L$ is a closed normal subgroup of $T$ such that the quotient group $T/L$ is a compact group.

We also will need a technical lemma.

\begin{lem}\label{lem5.2}
There exists a normal closed subgroup $N$ of $G$ such that $N\subseteq S$ and that $N$ is of finite index in $G$.
\end{lem}

\begin{proof}
Under the discrete topology of $G$, $S$ contains a normal subgroup, say $M$, of $G$ of finite index in $G$ (see \cite[p.~26]{HR} or \cite[p.~24]{G76}). Now let $N=\textrm{cls}_GM$.
Then it is easy to check that $N$ is a normal closed subgroup of $G$ satisfying the
claim of the lemma.
\end{proof}

From now on, let $N$ be as in Lemma~\ref{lem5.2}.
Let $(G,\Pi(G)), (S,\Pi(S))$ and $(N,\Pi(N))$ be the universal minimal proximal flows of the topological groups $G, S$ and $N$, respectively, as in (C) in $\S\ref{sec2.2}$.

First we need to extend the natural action of $N$ on $\Pi(N)$ to an action of $G$ on $\Pi(N)$. For this
we define an action of $G$ on $\Pi(N)$ by mapping $G$ into $\mathrm{Aut}(N)$ as follows:
\begin{equation*}
\zeta\colon G\times\Pi(N)\rightarrow\Pi(N);\quad (t,z)\mapsto\widehat{\sigma_{\!t}{\!\upharpoonright_{\!N}}}(z),
\end{equation*}
where\;
$\widehat{{}}\,\colon\mathrm{Aut}(N)\rightarrow\mathrm{Homeo}(\Pi(N))$
is the canonical map introduced in $\S\ref{sec3}$, with $N$ in place of $G$.



\begin{lem}\label{lem5.3}
$(G,\Pi(N),\zeta)$ is a flow, and as such it is minimal and proximal.
\end{lem}

\begin{proof}
We only need to verify $\widehat{\sigma_{\!t}{\!\upharpoonright_{\!N}}}(z)$ is jointly continuous with respect to $t\in G$ and $z\in\Pi(N)$.
Since $N$ is closed and of finite index in $G$, there exists a finite set, say $\{s_1,\dotsc,s_k\}\subseteq G$, such that $G=s_1N\cup\dotsm\cup s_kN$ is a clopen partition and $s_iN\cap s_jN=\emptyset$ for $1\le i\not=j\le k$.
Now suppose that the net $(t_i,z_i)\to(t,z)$ in $G\times\Pi(N)$.
Then, passing to subnets (and relabeling) we can assume that for some fixed
$j_0$ we have $t_i \in s_{j_0}N$ for every $i$.
Thus for every $i$ there is some $r_i \in N$ so that $t_i = s_{j_0}r_i$ and
$r_i \to r : = s_{j_0}^{-1}t$.
Observe that for $r \in N$ and $z \in \Pi(N)$
we have $\widehat{\sigma_{\!r}\!\!\upharpoonright_{\!N}}(z)= rz$.
Now, by the continuity of the $N$ action on $\Pi(N)$,
\begin{equation*}
\widehat{\sigma_{\!t_i}\!\!\!\upharpoonright_{\!N}}(z_i)
= \widehat{\sigma_{\!s_{j_0}r_i}\!\!\upharpoonright_{\!N}}(z_i)=
\widehat{\sigma_{\!s_{j_0}}\!\!\!\upharpoonright_{\!N}}
\left(\widehat{\sigma_{\!r_i}\!\!\upharpoonright_{\!N}}(z_i)\right)=
\widehat{\sigma_{\!s_{j_0}}\!\!\!\upharpoonright_{\!N}}
(r_i z_i)
\to \widehat{\sigma_{\!s_{j_0}}\!\!\!\upharpoonright_{\!N}}
(r z).
\end{equation*}
But
\begin{equation*}
\widehat{\sigma_{\!s_{j_0}}\!\!\!\upharpoonright_{\!N}}(rz)=
\widehat{\sigma_{\!s_{j_0}}\!\!\!\upharpoonright_{\!N}}
\left(\widehat{\sigma_{\!r}\!\!\upharpoonright_{\!N}}(z)\right) =
\widehat{\sigma_{\!s_{j_0}r}\!\!\upharpoonright_{\!N}}(z) =
\widehat{\sigma_{\!t}\!\!\upharpoonright_{\!N}}(z).
\end{equation*}
Thus $\zeta\colon G\times\Pi(N)\rightarrow\Pi(N)$ is continuous.
\end{proof}

\begin{remk}
One can relax here the assumption that $N$ has
finite index in $G$.
Assuming only that $N$ is a clopen normal subgroup will suffice.
Indeed, under this assumption the space $G/ N$ is discrete
and (denoting the
quotient map $Q\colon G \to G/N$) we have, in the notation of
the proof above, $Q(t_i) \to Q(t)$. Thus, eventually $Q(t_i) = Q(t)$ and we can assume that
for every $i$, $t_i = s_{j_0}r_i$ for some fixes $s_{j_0} \in G$ and $r_i \in N$.
Now proceed as before.
\end{remk}

With the above preparations at hand, we are ready to complete our proof of Theorem~\ref{thm2.7}.

\begin{proof}[Proof of Theorem~\ref{thm2.7}]
By the universality of $(G,\Pi(G))$, there exists a $(G,\Pi(G))\xrightarrow{\phi}(G,\Pi(N))$, where $(G,\Pi(N))=(G,\Pi(N),\zeta)$ as in Lemma~\ref{lem5.3}.

Next consider the flow $(N,\Pi(G))$ which is obtained by restricting the action of $G$ on $\Pi(G)$ to the action of its subgroup $N$. By Lemma~\ref{lem5.1} this flow is minimal and proximal;
therefore there is a homomorphism
$(N,\Pi(N))\xrightarrow{\psi}(N,\Pi(G))$.
Now $(N,\Pi(N))\xrightarrow{\phi\circ\psi}(N,\Pi(N))$ is an endomorphism, hence it is the identity map by Lemma~\ref{lem4.1}.
Therefore $(G,\Pi(G))\xrightarrow{\phi}(G,\Pi(N))$ is an isomorphism.

By Lemma~\ref{lem5.1} again, $(S,\Pi(G))$ is minimal and proximal,
hence there exists a homomorphism
$(S,\Pi(S))\xrightarrow{\theta}(S,\Pi(G))$.
Similarly, $(N,\Pi(S))$ is minimal and proximal and there exists a homomorphism $(N,\Pi(N))\xrightarrow{\lambda}(N,\Pi(S))$. Thus the composition
\begin{equation*}
(N,\Pi(N))\xrightarrow{\lambda}(N,\Pi(S))\xrightarrow{\theta}(N,\Pi(G))
\xrightarrow{\phi}(N,\Pi(N))
\end{equation*}
is an endomorphism of $(N,\Pi(N))$, hence it is the identity map.
Thus $(N,\Pi(N))\xrightarrow{\lambda}(N,\Pi(S))$ is an isomorphism.
Using this isomorphism, together with the
isomorphism
$(G,\Pi(G))\xrightarrow{\phi}(G,\Pi(N))$, an action of $G$ on $\Pi(S)$ can be defined so that $(G,\Pi(S))$ and $(G,\Pi(G))$ are isomorphic.
\end{proof}

%

It is well known that
for locally compact groups
every subgroup of an amenable group is amenable.
How about strongly amenable groups?

\begin{cor}\label{cor5.5}
Let $G$ be a topological group and $S$ a closed subgroup of finite index in $G$. Then $G$ is strongly amenable if and only if so is $S$.
\end{cor}

\begin{proof}
By Theorem~\ref{thm2.7}, it follows that $\Pi(G)\cong\Pi(S)$. So $\Pi(G)$ is a singleton if and only if so is $\Pi(S)$. This proves Corollary~\ref{cor5.5}.
\end{proof}


\begin{que}\label{q5.6}
Let $G$ be a locally compact group. If $G$ is a compact extension of $S$,
is it true that $(G,\Pi(G))\cong(G,\Pi(S))$?
\end{que}

If the answer to Question~\ref{q5.6} is positive, then as in Corollary~\ref{cor5.5}
we will conclude that $G$ is strongly amenable if and only if so is $S$.

\section{Commensurators}\label{Comm}
If $G$ is a group and $H<G$ a subgroup, we denote $H^g = gHg^{-1}$ and $H_g = H \cap H^g$.
The {\it commensurator of $H$ in $G$} is defined by
$$
\bH=\textrm{Comm}_G(H) =\{g \in G\,|\, H_g \ {\text{ has finite index in both}} \  H\  {\text{and}}\  H^g\}.
$$

As a corollary of \cite[Theorem II.4.4]{G76} we obtain the following.

\begin{thm}
Assume that $\bH$ is countable. Then,
the canonical action of $H$ on $\Pi(H)$ can be extended to an action
$(\bH, \Pi(H))$.
That is, there is a homomorphism $g \mapsto \hat{g}$, from $\bH$ to
${\mathrm{Homeo}}(\Pi(H))$,
such that
\begin{itemize}
\item
$\hat{g}$ satisfies the equation $\hat{g}(tx) = (gtg^{-1})\hat{g}(x)$ for every
$t\in H\cap g^{-1}Hg$ and $x\in\Pi(H)$.
\item
for $h \in H$ and $x \in \Pi(H)$ we have
$\hat{h}(x) = h(x)$.
\item
for every $g \in \bH$
$$
(H, \Pi(H)) \cong (H, \Pi(H_g)) \cong (H^g, \Pi(H_g)) \cong (H^g, \Pi(H^g)).
$$
Analogous statements hold for $\Pi_s$.
\end{itemize}
\end{thm}

\begin{proof}
Let $\{e= g_0, g_1, g_2, \dotsc, g_k, \dotsc\}$ be an enumeration of $\bH$ and for each $k$
let $H_k = \bigcap_{i=0}^k g_i H^{g_i}$.
Then, $H_k$ has finite index in $H$ and there is a
normal subgroup $N $ of $H$, with $N < H_k$,
such that $N$ is of finite index in $H$.
The flow $(N, \Pi(H))$ is minimal, proximal and for each $0 \le i \le k$,
the map
$\sigma_{\!g_i}\!\!\upharpoonright\!\!N$
is an automorphism of $N$.
Thus the corresponding map
$\widehat{\sigma_{\!g_i}\!\!\upharpoonright\!\!N}
\colon\Pi(H) \to \Pi(H)$ is a homeomorphism and
the map $\; \widehat{{}}\,\colon \langle g_0, g_1, g_2, \dotsc, g_k \rangle \to \mathrm{Homeo}(\Pi(H))$
is a group homomorphism.
Note that when $N_1 < N_2 < H$ and $g \in \bH$ normlizes both $N_1$ and $N_2$ we have, with $\widehat{g_j}$ the corresponding homeomorphisms induced by
$\sigma_{\!g_j}\!\!\!\upharpoonright_{\!N_j}, \ j=1,2$,
$$
\widehat{{g_2}^{-1}} (\widehat{g_1}(tx)) =
\widehat{{g_2}^{-1}} (g t g^{-1} \widehat{g_1}(x))=
t\widehat{{g_2}^{-1}}(\widehat{g_1}(x)),
$$
for every $t \in N_1$. Thus $\phi = \widehat{{g_2}^{-1}}\circ \widehat{g_1}$
is an automorphism of the minimal proximal flow $(N_1, \Pi(H))$, whence
$\widehat{{g_2}^{-1}}\circ \widehat{g_1}=\textit{id}_{\Pi(H)}$, and
$\widehat{{g_2}} = \widehat{g_1}$.

We now let $k\to\infty$ to conclude the proof.
\end{proof}

With $G = \textrm{SL}(n, \mathbb{R})$ and $H=\textrm{SL}(n,\mathbb{Z})$
we have $\textrm{Comm}_G(H) = \textrm{SL}(n, \mathbb{Q})$ and
a nice instance of this theorem is the result that the action
$(\textrm{SL}(n,\mathbb{Z}), \Pi(\textrm{SL}(n,\mathbb{Z})))$ extends to an action
$(\textrm{SL}(n, \mathbb{Q}), \Pi(\textrm{SL}(n,\mathbb{Z})))$.
Again an analogous result is valid for $\Pi_s$ instead of $\Pi$.

\section*{Acknowledgments}
X.~Dai was partly supported by National Natural Science Foundation of China (Grant Nos. 11431012, 11271183) and PAPD of Jiangsu Higher Education Institutions; and E.~Glasner was supported by a grant of the Israel Science Foundation (ISF 668/13).



\end{document}